\documentclass[11pt, reqno]{amsart}

\usepackage{amssymb,amsmath,amscd,graphicx,
latexsym,amsthm}
\usepackage{amssymb,latexsym,amsmath,amscd,graphicx}
\usepackage[mathscr]{eucal}

\usepackage{scalerel,stackengine}
\DeclareMathSymbol{\shortminus}{\mathbin}{AMSa}{"39}

\stackMath
\newcommand\reallywidehat[1]{%
\savestack{\tmpbox}{\stretchto{%
  \scaleto{%
    \scalerel*[\widthof{\ensuremath{#1}}]{\kern.1pt\mathchar"0362\kern.1pt}%
    {\rule{0ex}{\textheight}}%WIDTH-LIMITED CIRCUMFLEX
  }{\textheight}% 
}{2.4ex}}%
\stackon[-6.9pt]{#1}{\tmpbox}%
}
 \input xy
 \xyoption{all}
\def\OO{{\mathcal O}}

\def\LL{\mathcal{L}}

\def\LL{\mathcal{L}}

\def\F{\mathcal{F}}

\def\G{\mathcal{G}}

\def\J{\mathcal{J}}

\def\cP{\mathcal{P}}
\def\Pic0{{\rm Pic}^0}

\def\Aut0{{\rm Aut}^0}

\def\Z{\mathcal Z}
\def\T{\mathcal T}

\def\*{{\underline *}}
\def\utheta{{\underline \theta}}

\def\Alb{{\rm Alb}\,}

\def\ev{\mathrm{ev}}

\DeclareMathOperator{\codim}{codim}

\theoremstyle{plain}

\newtheorem{theorem}{Theorem}[subsection]
\newtheorem{theoremalpha}{Theorem}

\newtheorem{proposition/example}[theorem]{Proposition/Example}
\newtheorem{proposition}[theorem]{Proposition}

\newtheorem{corollary}[theorem]{Corollary}
\newtheorem{lemma}[theorem]{Lemma}

\theoremstyle{definition}

\newtheorem*{introdefinition}{Definition}
\newtheorem{definition}[theorem]{Definition}

\newtheorem{remark}[theorem]{Remark}
\newtheorem{example}[theorem]{Example}

\newtheorem{conjecture/question}[theorem]{Conjecture/Question}

\newtheorem{remark/definition}[theorem]{Remark/Definition}
\newtheorem{notation/assumptions}[theorem]{Assumptions/Notation}
\newtheorem{setting/notation}[theorem]{Setting and Notation}

\numberwithin{equation}{section}

 \theoremstyle{remark}

\begin{document}
\title[Generation and ampleness of coherent sheaves on abelian varieties]
{Generation and ampleness of coherent sheaves on abelian varieties, with application to Brill-Noether theory
 }
 \author[G.Pareschi]{Giuseppe Pareschi}

\address{Dipartimento di Matematica,
              Universit\`a di Roma Tor Vergata\\Italy}
\email{pareschi@mat.uniroma2.it}
 \thanks{Partially supported by  the MIUR Excellence Department Project MatMod@TOV awarded to the Department of Mathematics, University of Rome Tor Vergata.}

%\subjclass[2010]{14K25; 32G20}
\begin{abstract}  We introduce a variant of  global generation for coherent sheaves on abelian varieties which, under certain circumstances, implies ampleness. This extends a criterion of Debarre asserting that a continuously globally generated coherent sheaf on an abelian variety is ample. We apply this to show the ampleness of certain sheaves, which we call naive Fourier-Mukai-Poincar\'e transforms, and to  study the structure of GV sheaves. In turn, one of these applications allows to extend  the classical existence and connectedness results of Brill-Noether theory to a wider context, e.g. singular curves equipped with a suitable morphism to an abelian variety. Another application is a general inequality of Brill-Noether type involving the Euler characteristic and the homological dimension.
\end{abstract}
%\dedicatory{To Enrico Arbarello}
\maketitle

\section{Introduction}\label{intro} 

%%%%

\subsection{Motivation: continuous global generation, M-regularity, and ampleness }\label{intro-CGG} We work with projective varieties over an algebraically closed field $k$. On abelian (or, more generally, irregular) varieties, the basic notion of global generation of a coherent sheaf (at a given point) admits a  variant, introduced by M.~Popa and the author, called \emph{continuous global generation} (CGG for short), see  \cite[Definition 2.10]{reg1}.  This is as follows: given a coherent sheaf $\F$ on an abelian variety $A$, a closed point $x\in A$, and a Zariski open set $U\subset \Pic0 A$, one considers the sum of twisted evaluation maps
\begin{equation}\label{evU}
 \ev_U(x):\bigoplus_{\alpha\in U}H^0(A,\F\otimes P_\alpha)\otimes P_\alpha^{-1}\rightarrow \F_{|x}\>,
 \end{equation}
where: $\F_{|x}:=\F\otimes k(x)$ denotes  the fiber of $\F$ at the point $x$, and $P_\alpha=\cP_{|A\times\{\alpha\}}$ denotes the line bundle on $A$ parametrized by a point $\alpha\in \Pic0 A$ via the (normalized) Poincar\`e line bundle $\cP$ on $A\times \Pic0 A$.  

The sheaf  $\F$ is said to be CGG at $x$ if this map is surjective \emph{for all} (non empty) Zariski open subsets of $\Pic0 A$. Of course, if this holds for every $x\in A$ the sheaf $\F$ is said to be CGG.\footnote{There is a more general version of these notions and their consequences holding for sheaves on any  variety equipped with a morphism to an abelian variety, see Definition \ref{def:irreg}. In this paper we mostly consider sheaves on abelian varieties.}
 
  The condition of being CGG   neither implies nor is implied by  the usual \emph{global generaton} (GG). For example, the line bundle associated to a theta divisor on a p.p.a.v. is CGG but  not GG and the structure sheaf of an abelian variety is GG but  not CGG. However the following three facts make the CGG property significant:
  
  \noindent  {\bf (i)} \emph{Ampleness.}   Let $\mu_N:A\rightarrow A$ be the multiplication by $N$. If $\F$ is a CGG coherent sheaf on $A$  then there is a $N\gg 0$ such that $\mu_N^*(\F\otimes P_\alpha)$ is GG \emph{for all} $\alpha\in\Pic0 A$. Consequently, a CGG sheaf on an abelian variety is ample (\cite[Proposition 3.1 and Corollary 3.2]{debarre}).\footnote{(a) The notion of ampleness have been extended to coherent sheaves by Kubota \cite{kubota}, and many properties holding for vector bundles extend to this setting, see \cite{debarre}.\\
  (b) The quoted result is stated in \emph{loc cit} over the complex numbers, but the proof works over any algebraically closed field.}

  \noindent {\bf (ii)}  \emph{Criterion for global generation.}  If $\F$ and $L$ are respectively a CGG coherent sheaf on $A$, and a CGG line bundle on a subvariety of $A$, then $\F\otimes L$ is GG (this follows by considering sections of the form $s_\alpha\cdot t_{-\alpha}$ with $s_\alpha\in H^0(\F\otimes P_\alpha)$ and $t_{-\alpha}\in H^0(L\otimes P_\alpha^{-1})$, \cite[Proposition 2.12]{reg1}).
  
  \noindent {\bf (iii)}  \emph{Cohomological criterion.}  There is a natural vanishing condition on the higher cohomology, named\emph{ M-regularity}, implying CGG (\cite{reg1}).

This package has found various applications. To name a few: effective projective normality and  syzygies of abelian varieties (e.g. \cite{survey1},\cite{ito}); effective birationality of pluricanonical maps of irregular varieties (e.g. \cite{survey1}, \cite{jlt}, \cite{bicanonical}, \cite{jiang-sun}, \cite{cinesi}); positivity of direct images of pluricanonical sheaves  of varieties mapping to abelian varieties (e.g. \cite{debarre},  \cite{cy}, \cite{loposc}, \cite{M1},\cite{M2}, \cite{meng-popa});  birational geometry and volume of irregular varieties via eventual maps (e.g. \cite{zhi}).

The  property of being CGG, as well as the M-regularity condition implying it, are quite strong and not often verified.  In this paper we introduce the following weaker condition with the purpose of  enlarging  the range of applicability of this circle of ideas,  especially  Debarre's criterion for ampleness mentioned in {\bf (i)}.\footnote{In the same spirit,  a related notion, namely \emph{weak CGG}, was introduced by M. Popa and the author in \cite[Definition 2.1]{reg2}, but the notion given here is better behaved.}

%%%%

 \subsection{Generation by a set of subvarieties of the dual abelian variety }\label{intro-gen}  
 \begin{introdefinition}
 Given a finite set of irreducible subvarieties of $\Pic0 A$, say \ $\Z=\{Z_i\}$ (such that no subvariety  is contained in another one),  a coherent sheaf $\F$ on $A$ is said to be \emph{generated by $\Z$} (at a given point $x\in A$) if the map (\ref{evU}) is surjective \emph{for all open sets $U$ of $\Pic0 A$  such that $U$ meets all subvarieties $Z_i\in\Z$.} 
 \end{introdefinition}
 In this language, to  be CGG means to be generated by the trivial set $\Z=\{\Pic0 A\}$. On the other hand, a GG sheaf is generated by the set consisting the origin alone: $\Z=\{\hat 0\}$ (but the converse does not hold in general). We say that a coherent sheaf is \emph{generated} if it is generated by some set of subvarieties (this  coincides with the notion of algebraic generation of \cite[Definition 3.2]{lin-yu}). 
  %It is easily seen that this is equivalent to the surjectivity of  the map
%  \[
%\bigoplus_{\alpha\in \Pic0 A}H^0(A,\F\otimes P_\alpha)\otimes P_\alpha^{-1}\rightarrow \F\>. \]

 There is a natural notion of \emph{irredundant} generating set (see Remark \ref{irredundant}),  and it can be shown that, for any  finite set of subvarieties $\Z=\{Z_i\}$ as above, there are generated sheaves $\F$ having  $\Z$ as irredundant generating set (Example \ref{any} below).

  Next, an irreducible subvariety $Z$ of an abelian variety  $B$ is said \emph{to span}  $B$ if $\overbrace{Z+\cdots+Z}^N=B$ for some $N$. A finite set of irreducible subvarieties $\mathcal Z=\{ Z_i\}$ as above is said to \emph{strongly span} $B$  if \emph{each subvariety $Z_i$ spans $B$. }

Items {\bf (i)} and {\bf (ii)} of the previous subsection generalize as follows:
 
 \begin{theoremalpha}[Nefness/ampleness]\label{A}  If $\F$ is generated  then it is nef. If $\F$ is generated by a set of subvarieties $\mathcal Z$ strongly spanning $\Pic0 A$ then $\F$ is  ample \emph{(Theorem \ref{ample-criterion}).}
 \end{theoremalpha}
 
 \begin{theoremalpha}[Non-generation locus]\label{B}  If $\F$ is generated by a set $\Z=\{Z_i\}$ and  $L$ is a CGG line bundle on a subvariety $X$ of $A$, then the non-generation locus of $\F\otimes L$ \emph{(namely the locus $B(\F\otimes L)$ where $\F\otimes L$ is not globally generated)} is explicitly controlled in function of  the base loci of the line bundles $L\otimes P_\alpha$, with $\alpha\in Z_i$  \emph{(Proposition \ref{base-locus})}. \emph{ (As expected, it turns out that the bigger are the $Z_i$, the smaller is $B(\F\otimes L)$).}
 \end{theoremalpha}

%%%
 
 \subsection{Relationship with the Fourier-Mukai-Poincar\'e transform }\label{relation} We are left with the problem of finding  criteria or applicable methods for proving that a given coherent sheaf is generated. In view of Theorem \ref{A}, it is especially relevant to understand whether a given sheaf $\F$ is generated by some set of subvarieties strongly spanning $\Pic0 A$. The author hopes that this approach will be helpful in addressing the ampleness of vector bundles on subvarieties of abelian varieties.
 
The next result is a step in this direction. Namely we provide a sufficient condition (close to be a characterization) ensuring that a given coherent sheaf $\F$ is generated, and we identify the irredundant set of subvarieties doing the job. Roughly speaking, this is a  subset of the set of irreducible components of the supports of the sheaves appearing in the torsion filtration (\cite[Definition 1.1.4]{huybrechts-lehn}) of a certain coherent sheaf on $\Pic0 A$ associated to $\F$, referred to as  \emph{the naive Fourier-Mukai-Poincar\'e (FMP) transform of $\F$}. However, except for some cases, this is just a first step toward the solution of the above problem, since in general the torsion filtration of the naive FMP transform of $\F$ is not easy to describe.
 
 Passing to a more detailed description, we  will consider the FMP functor in the following form. Let $\cP$ be the normalized Poincar\`e line bundle on $A\times\Pic0 A$. We consider the Fourier-Mukai equivalence
 \begin{equation}\label{FM}\Phi^A_{\cP^{-1}}: D(A)\rightarrow D(\Pic0 A), \qquad \Phi^A_{\cP^{-1}}(\>\cdot\>)=\mathbf Rq_*(p^*(\>\cdot\>)\otimes \cP^{-1})
  \end{equation}
    where $p$ and $q$ are the  projections on $A$ and $\Pic0 A$. 
  We will also make use of the dualization  functor in the following (unshifted) form:
\begin{equation}\label{dual}\F^\vee:=\mathbf R\mathcal{H}om(\F,\OO_A).
\end{equation}

The above mentioned  \emph{naive FMP transform} of a given coherent sheaf $\F$ on $A$ is defined as follows
\begin{equation}\label{T(F)}
\T(\F):=R^g\Phi^A_{\cP^{-1}}(\F^\vee).
\end{equation}
where $g=\dim A$. By Serre-Grothedieck duality, $H^i(A,\F^\vee\otimes P_\alpha^{-1})=0$ for all $i>g$ and for all $\alpha\in \Pic0 A$. Therefore, by base change and duality, we have the canonical isomorphisms
\[R^g\Phi^A_{\cP^{-1}}(\F^\vee)_{|\alpha}\cong H^g(A, \F^\vee\otimes P_\alpha^{-1})\cong H^0(A,\F\otimes P_\alpha)^\vee
\]
for all $\alpha \in \Pic0 A$ (note that the above $H^g$ is usually a hypercohomology group). Hence (up to duality) the coherent sheaf $\T(\F)$ encodes the variation of the $k$-linear spaces $H^0(A,\F\otimes P_\alpha)$ as $\alpha$ varies in $\Pic0 A$. Therefore it is natural to expect that $\T(\F)$ must be related somehow to the generation of the coherent sheaf $\F$. It turns out that it is the \emph{torsion} of the sheaf $\T(\F)$ what really matters.

 Of course, in general it is not the case  that the naive transform coincides (up to shift) with the  transform in the derived category, namely  that
\begin{equation}\label{GV-sheaf}
\Phi^A_{\cP^{-1}}(\F^\vee)=R^g\Phi^A_{\cP^{-1}}(\F^\vee)[-g].
\end{equation}
In fact the sheaf $\F$ is said to be a \emph{GV sheaf} (generic vanishing sheaf) precisely when (\ref{GV-sheaf}) holds.

The precise relation of the FMP transform with the generation problem is stated in Corollaries \ref{gen-FM} and \ref{generated}. It could be somewhat  imprecisely summarized as follows 

\begin{theoremalpha}[Generation  of sheaves and the FMP transform]\label{C}  (a) The surjectivity of the map (\ref{evU}) can be rephrased in terms of a condition involving both the FMP transform $\Phi^A_{\cP^{-1}}(\F^\vee)$ and the naive FMP transform $R^g\Phi^A_{\cP^{-1}}(\F^\vee)$.

\noindent (b) If $\F$ is generated, then 
a certain subset of the set of  irreducible components of the supports of the sheaves appearing in the torsion filtration
 of the naive FMP transform of $\F$ form an (irredundant) generating set for $\F$ . 
 \end{theoremalpha}

As mentioned above, the problem in applying condition (b) is that at present there is no general way of describing, in terms of $\F$, the supports of the sheaves appearing in the torsion filtration  
of the naive FMP transform of $\F$  (only for GV sheaves there we have such a description, see Remark \ref{GV-torsion}).

% The hypotheses alluded to in {\bf (6)} are not easy to check, but they are automatically satisfied in certain significant cases (for example, sheaves which are already known to be GG, e.g. the normal and cotangent bundle of a subvariety of an abelian variety, and their symmetric products). However, even in these cases the main problem is to describe the sheaves appearing in the torsion filtration of $\T(\F)$ in function of $\F$. This is well known if $\F$ is a GV sheaf (because, if this is the case, the sheaf $\T(\F)$ coincides, up a shift, with the full FMP transform of $\Phi_{\cP^{-1}}^Q(\F^\vee)$) but we known have an answer in general. 

\subsection{Ampleness of naive FMP transforms (= generalized Picard bundles) }\label{pic}  The next result concerns a lucky class of coherent sheaves which turn out to be  generated and ample, even when they are not GV sheaves: the naive FMP transforms themselves. More generally, the same holds for the naive FMP transforms of coherent sheaves on certain reduced   schemes mapping to abelian varieties. These are defined similarly, with the difference that the FMP functor is not an equivalence anymore (it turns out that, for the specific result described in the present subsection, this is unnecessary).

 Specifically, given an equidimensional, reduced, Cohen-Macaulay projective scheme, equipped with a morphism to an abelian variety $f: X\rightarrow A$, one defines $\cP_X=(f\times {\rm id})^*\cP$ and considers the functor (not an equivalence anymore)
\[\Phi^X_{\cP_X^{-1}}:D(X)\rightarrow D(\Pic0 A)\]
and the (unshifted) dualization functor
\[\Delta_X: D(X)\rightarrow D(X), \qquad \Delta_X(\F)={\bf R}\mathcal Hom(\F,\omega_X).
\]
The naive FMP transform of a coherent sheaf $\F$ on $X$ is defined as
\[\T(\F)=R^d\Phi^X_{\cP_X^{-1}}(\Delta_X(\F)),
\]
where $d=\dim X$, and has the same meaning as discussed in the previous subsection.

We will consider coherent sheaves $\F$  on a reduced scheme $X$ as above  with the property that all subsheaves of $\F$  have reduced scheme-theoretic support (equivalently, one can consider only the subsheaves appearing in the torsion filtration of $\F$). 
We will adopt the following notation: given a sheaf $\F$ on $X$, we denote $\Z(\F)=\{Z_i\}$ the set of irreducible components of the supports of the sheaves on $X$ appearing in the torsion filtration of $\F$, and $f(\Z(\F)):=\{f(Z_i)\}$.

\begin{theoremalpha} [Generation and ampleness of naive FMP transforms]\label{D}  In the above setting, let $\F$ be a coherent sheaf on a scheme $X$ as above, such that all subsheaves of $\F$ have reduced scheme-theoretic support \emph{(the simplest example are torsion free sheaves on $X$)}.
 Then the naive FMP transform $\T(\F)$ is generated by a subset of the collection $f(\Z(\F))$. \\
  In particular, if each subvariety of the set $\Z(\F)$ maps, via the morphism $f$, to a subvariety spanning the abelian variety $A$, then $\T(\F)$ is an ample sheaf on $\Pic0 A$. 
 \end{theoremalpha}
 
The above theorem is in fact is a generalization of a result of Schnell, proving that naive FM transforms of torsion free sheaves on abelian varieties  are CGG (\cite[Theorem 4.1]{schnell}).\footnote{Of course Schnell does not use this terminology. Moreover his result is stated in the language of the \emph{symmetric Fourier transform}, introduced in the same paper.} The argument used in the proof is  essentially Schnell's.

We remark that, although  in a different language, a well known example of Theorem \ref{D} is the ampleness of (dual) \emph{Picard bundles}: here $X$ is a smooth complex projective variety, equipped with its Albanese morphism $a:X\rightarrow \Alb X$, and $\F=\LL$ is a line bundle on $X$ satisfying the following vanishing condition on the higher cohomology: 
 \begin{equation}\label{IT(0)Pic}
 H^i(X,\LL\otimes P_\alpha)=0
 \end{equation} for all $\alpha\in\Pic0 X$ and $i>0$.\footnote{When dealing with sheaves on abelian varieties this condition is usually referred to as the \emph{IT(0) condition} (namely $\F$ satisfied the index theorem with index $0$).}   By base change this ensures that $\T(\LL)$ is a vector bundle on $\Pic0 X$. In this case the generating collection is given by a single subvariety, namely the Albanese image $\{a(X)\}$, which automatically spans. By Theorem \ref{D},  $\T(\LL)$ is ample. The dual of $\T(\LL)$, namely $R^0\Phi^X_{\cP_X}(\F)$, 
 is called the Picard bundle associated to the line bundle $\LL$.\footnote{It is easily seen that this definition is equivalent to the one of \cite[\S 6.3.C]{laz2}, where the Picard bundle is defined as a vector bundle on ${\rm Pic}^\lambda X$, where $\lambda$ is the algebraic equivalence class of $L$.} Picard bundles were classically known to be be negative for $\dim X=1$ (\cite{acgh}). This result was subsequently generalized by Lazarsfeld to all smooth projective varieties (\cite[\S 6.3.C]{laz2}). Thus Theorem \ref{D} can be seen as a vast generalization (with a completely different proof) of that result.  
 
 It is worth to remark  that in Theorem \ref{D} we are not assuming condition (\ref{IT(0)Pic}), nor any other vanishing conditions. Interestingly, there are many examples of coherent sheaves not verifying (\ref{IT(0)Pic}) such that nevertheless  their naive FMP transform is locally free (see Remark \ref{loc-free}). 
 
\subsection{Application to Brill-Noether theory of singular curves}  One immediate application of Theorem \ref{D} is to Brill-Noether theory of  (complex) \emph{singular curves \emph{(even reducible and with non-planar singularities)} equipped with a morphism to an abelian variety $A$ such that the image of each component spans $A$}. In this setting we provide analogues  of the existence and connectedness theorems for special divisors, Ghione and Segre-Nagata theorem. We refer to Subsection \ref{sub:BN} for these results.

  %%%
  
   \subsection{A general inequality of Brill-Noether type} Another application of Theorem \ref{D} is a general existence result of Brill-Noether type which, although somewhat weak, is optimal, and turns out to be useful in some applications.  In this subsection we will assume that the ground field is $\mathbb C$. Given a coherent sheaf $\F$ on a complex abelian variety $A$,\footnote{also in this case one could extend the discussion of this topic to varieties equipped of a \emph{finite} morphism  to an abelian varieties, but for sake of simplicity we will stick to abelian varieties} one defines the  cohomological support loci
 \begin{equation}\label{Vi}
 V^{i}(\F)=\{\alpha\in \Pic0 A\>|\>h^i(\F\otimes P_\alpha)>0\}
 \end{equation}
 and 
 \begin{equation}\label{V>0}
 V^{>0}(\F)=\bigcup_{i>0}V^i(\F).
 \end{equation}
 \begin{theoremalpha}\label{cast-de-franchis} Let $\F$ be a coherent sheaf on a complex abelian variety $A$ such that all its torsion subsheaves (if any) have reduced scheme-theoretic support, and each component of the support spans $A$ \emph{(simplest example: a torsion free sheaf on $A$)}. Assume that condition (\ref{IT(0)Pic}) holds, i.e. 
 \[ V^{>0}(\F)=\emptyset .
 \]
 Then 
 \[\chi(\F)\ge{\rm hd}(\F)+1.\]
 \end{theoremalpha}
Here ${\rm hd}(\F)$ denotes the homological dimension.
% Vedi gli esercizi su hartshorne (probabilmente presi da ega o sga). In particolare, hd e' il sup delle pd locali in x. a loro volta le pd locali sono il sup degli i tali che ext^i(\F,A) e' non nullo (perche' Ae' un anello locale regolare). Pero' a me non torna tanto che la hd sia il sup delle risoluzioni localmente libere globali. 
Easy examples showing that (at least  the way it is stated) Theorem \ref{cast-de-franchis} is optimal, are: 

\noindent (a) locally free sheaves $\F$ on abelian varieties (of arbitrarily high rank) such that $V^{>0}(\F)=\emptyset$ and $\chi(\F)=1$. (It is known that such sheaves are of the form  $\F=\Phi_{\cP^{-1}}(L^\vee)$, where $L$ is any ample line bundle on $A$. Then $\chi(\F)=1$ and ${\rm rk}(\F)=\chi(L)$.)

\noindent (b) Let $i:C\rightarrow J(C)$ a Abel-Jacobi embedding of a curve in its Jacobian and let $\F=i_*L$, where $L$ is a line bundle on $C$. In this case $V^{>0}(\F)=V^1(\F)$, which  is non-empty if and only if $\deg (L)\le 2g-2$, i.e. $\chi(\F)\le g-1={\rm hd}(\F)$.
 
An example of application of Theorem \ref{cast-de-franchis} is as follows. Let  $\J(D)$ be  the multiplier ideal sheaf of an effective $\mathbb Q$-divisor $D$ on an abelian variety $A$, and let $L$ be an ample line bundle  on $A$ such that $L-D$ is nef and big. By Nadel's vanishing the sheaf $\J(D)\otimes L$ satisfies the assumptions of Theorem \ref{cast-de-franchis}. Letting $Z$ the scheme of zeroes of $\J(D)$, it follows that 
 \begin{equation}\label{beppe}\chi(\J(D)\otimes L)\ge{\rm codim}\,Z
 \end{equation}
  (meaning the maximal codimension of a component of $Z$). 
 
 This  is instrumental in the proof  of a result of the author  on singularities of divisors on  simple complex abelian varieties (\cite{sing}, see also Corollary \ref{nadel}).

 %%%%% 
  
  \subsection{The case of GV sheaves}  GV sheaves are those sheaves such that their naive FMP transform $\T(\F)$ coincides, up to shift, with the FMP transform in the derived category (see (\ref{GV-sheaf})).\footnote{If this is the case the sheaf $\T(\F)$ is often denoted $\widehat{\F^\vee}$ in the literature.}  They are very special cases of naive FMP transforms since it follows from the inversion formula for the FMP equivalence that
  \begin{equation}\label{invertible} \F=\T(\T(\F))
  \end{equation}
  (see Proposition \ref{T(T)}). Therefore part (a) of the next result, on the generation and ampleness of such sheaves, follows from Theorems \ref{C} and {\ref{D}. As M regular sheaves form a subclass of GV sheaves, this recovers as a particular case the M-regularity criterion ({\bf iii}) of Subsection \ref{intro-gen}.  Item (b) is a structure result for such GV sheaves, following from the analysis of the torsion filtration of the FMP transform.   Loosely speaking, the content is that the building blocks for GV sheaves are either CGG (hence ample) sheaves, or sheaves of the form $(p^*\G)\otimes P_\alpha$, where $p: A\rightarrow B$ is a surjective homomorphism with connected kernel onto a lower dimensional abelian variety and $G$ is a CGG (hence ample) sheaf on $B$.
  %This  is as an extension of the M-regularity criterion {\bf (iii)} of Subsection \ref{intro-CGG}. In fact, by definition,  M-regular sheaves are GV sheaves such that their FMP transform $\T(\F)=R^g\Phi^A_{\cP^{-1}}(\F^\vee)=\Phi_{\cP^{-1}}^A(\F^\vee)[g]$ (recall that in this case the naive transform coincides with the transform in the derived category) is a torsion free sheaf. 
    \begin{theoremalpha}[see Theorem \ref{GV-ample}]\label{F} Let $\F$ be a GV sheaf on a $g$-dimensional abelian variety $A$, and assume  that all subsheaves of the FMP transform $\T(\F)$ have reduced support. Then:\\
  (a) $\F$ is generated and an irredundant generating set can be explicitly described;\\
  (b) $\F$ has a cofiltration whose kernels have in turn a cofiltration whose kernels are dominated either by ample sheaves or by sheaves of the form $(p^*\G)\otimes P_\alpha$ as above. 
  \end{theoremalpha}
  
 Concerning (a), let us recall that not all GV sheaves are generated. This is shown by the well known  example of  non-trivial \emph{unipotent} vector bundles on abelian varieties (see Example \ref{unipotent}). In fact that the FMP transform of a non-trivial unipotent vector bundle is a torsion sheaf scheme-theoretically supported at a non-reduced zero-dimensional scheme (set-theoretically supported at the origin of $\Pic0 A$). 
 
  It is worth to recall that a generation criterion for GV sheaves  was already given by Popa and the author (\emph{WIT criterion}, \cite[Theorem 4.1]{reg2}). Item (a) is a more precise version of that result. 
  In turn, item (b) can be seen as a weak analog of the \emph{Chen-Jiang decomposition} (\cite{cy}), a quite useful property satisfied by direct images of pluricanonical sheaves under morphisms to abelian varieties (see  Subsection \ref{structure}).

 %Thanks to Mihnea Popa for conversations and collaboration, a long time ago, on themes related to those of this paper. More recently, I benefited from the reading of Christian Schnell's interesting paper \cite{schnell}. 

% This paper is dedicated to Enrico Arbarello, with  admiration and gratitude.

 %%%%%%%%%

\section{Generated and ample coherent sheaves on abelian varieties}\label{gene}

In this section we provide some generalities on the notion of generation of coherent sheaves on abelian varieties introduced in Subsection \ref{intro-gen}, and prove Theorems \ref{A} and \ref{B}. 

%%%%%%%%%

\subsection{Generalities}   To begin with, we remark that, although the focus of this paper is on coherent sheaves on abelian varieties, one can extend the definition given in Subsection \ref{intro-gen} to the following setting: let $X$  be a projective variety, equipped with a morphism to an abelian variety
\[ f:X\rightarrow A. 
\]
\begin{definition}\label{def:irreg} Let  $\Z=\{Z_i\}_{i\in I}$ be a finite set of irreducible subvarieties of $\Pic0 A$, such that no subvariety belonging to $\Z$ is contained in another. A coherent sheaf $\F$ on $X$ is said to be \emph{generated by $\Z$ at  a given point $x\in X$} (with respect to te morphism $f$)  if the map
\[
{\rm ev}_U(x): \bigoplus_{\alpha\in U}H^0(\F\otimes f^* P_\alpha)\otimes f^* P_{\alpha}^{-1}\rightarrow \F_{|x}
\]
is surjective for all  open sets $U$ such that $U$ meets $\Z_i$ for all $i$. If this happens for all $x\in X$, namely the map
\[
{\rm ev}_U: \bigoplus_{\alpha\in U}H^0(\F\otimes f^* P_\alpha)\otimes f^* P_{\alpha}^{-1}\rightarrow \F
\] 
is surjective,  the sheaf $\F$ is said to be  generated by the set $\Z$ (with respect to the morphism $f$). If $\Z=\{\Pic0 A\}$ then $\F$ is said to be \emph{continuously globally generated} (CGG). 
Moreover  $\F$ is said to be \emph{generated} (with respect to the morphism $f$) if there is some set of subvarieties $\mathcal Z$ generating $\F$  (with respect to the morphism $f$). This is equivalent to the surjectivity of the map \begin{equation}\label{totalsum}
 {\rm ev}_{\Pic0 A}:\bigoplus_{\alpha\in\Pic0 A}H^0(\F\otimes f^*P_\alpha)\otimes f^*P_{\alpha}^{-1}\rightarrow \F.
 \end{equation}
\end{definition}

\begin{remark}\label{irredundant}    A finite set of subvarieties $\mathcal Y=\{Y_j\}$ as above is said to be \emph{covered} by  another  $\mathcal Z=\{ Z_i\}$ if for all $i$ there is a $j$ such that $Y_j$ is contained in $Z_i$. If a sheaf $\F$ is generated by a set $\mathcal Z$ then, trivially, it is generated by any set covered by $\mathcal Z$, and, as it will be clear in a moment, it turns out to be useful to identify a maximal (with respect to the relation of being covered) set doing the job. Such a set of subvariety will be called an \emph{irredundant generating set for $\F$.}
\end{remark}

\begin{remark}\label{paraculo}  In practice  in some arguments it may happen to find generating sets $\Z=\{Z_i\}$ such that  $Z_j$ is contained in $Z_k$ for some $j$ and $k$. However this does not cause any problem because, if this is the case, to be generated by the set $\Z$ is the same thing of being  generated by the set $\mathcal Z\smallsetminus \{Z_k\}$. Hence, by taking the subset of minimal subvarieties belonging to set $\Z$, one can always reduce to the assumption of Definition \ref{def:irreg}. 
\end{remark}

\begin{remark}\label{finite sum} By noetherianity and quasi-compactness, Definition \ref{def:irreg} can be equivalently formulated replacing the map ${\rm ev}_U$ of the Definition \ref{def:irreg} with the sum of finite number of evaluation maps. The required condition is the existence of positive integer $N_0$ and a collection of positive integers $\{N_i\}_{i\in I}$ such that the sum of twisted evaluation maps 
\[\bigoplus_{i\in I\cup\{0\}, 1\le j\le N_i}H^0(\F\otimes f^* P_{\alpha_{i,j}})\otimes f^* P_{\alpha_{i,j}}^{-1}\rightarrow \F
\]
is surjective for all sufficiently general $(\alpha_{0,1},\dots \alpha_{0,N_0})\in (\Pic0 A)^{N_0}$ and $(\alpha_{i,1},\dots, \alpha_{i,N_i})\in (Z_i)^{N_i}$ for all $i\in I$. Therefore also in the sum (\ref{totalsum}) one can take a finite number of summands.  In particular, the notion of\emph{ generated} coherent sheaf introduced here coincides with the the notion of \emph{algebraically generated} coherent sheaf introduced in the recent paper \cite[Definition 3.2]{lin-yu}. 

It follows that a coherent generated (with respect to any morphism) sheaf is  nef, as it is the quotient of the direct sum of numerically trivial line bundles. 
\end{remark}

 %%%%%%%
 
 \subsection{Theorem \ref{A}}
The main point is in the following easy lemma.
  
  \begin{lemma}\label{lem:sum} In the setting of Definition \ref{def:irreg}, let $\mathcal Z=\{Z_i\}_{i=1}^n$ and $\mathcal Y=\{Y_j\}_{j=1}^m$ be two finite sets of subvarieties of $\Pic0 A$. Let moreover $\F$ and $\G$ be coherent sheaves on $X$ respectively generated by $\mathcal Z$ and $\mathcal Y$ (with respect to the morphism $f$). Then $\F\otimes \G$ is generated by the set of subvarieties \emph{(see Remark \ref{paraculo})}:
  \[\mathcal{Z}+\mathcal{Y}:=\{Z_i+Y_j\}_{(i,j)=(1,1)}^{(n,m)}.
  \]
  \end{lemma}
  Here $Z_i+Y_j$ denotes, as usual, the image of $Z_i\times Y_j$ via the group law of $\Pic0 A$.
  
  \begin{proof} For open subsets of $\Pic0 A$, say $U$ and $V$, respectively meeting all the $Z_i$'s and  all the $Y_i$'s we have the surjective map 
\[
\begin{matrix}\bigl(\bigoplus_{\alpha\in U}H^0(A,\F\otimes f^*P_\alpha\bigr)\otimes f^*P_\alpha^{-1})\otimes\bigl( \bigoplus_{\beta\in V}H^0(A,\G\otimes f^*P_\beta)\otimes f^*P_\beta^{-1}\bigr)&\longrightarrow&\F\otimes \G\\
\Vert\\
\>\>\>\> \>\>\>\>\>\>\>\bigoplus_{(\alpha,\beta)\in U\times V}H^0(\F\otimes f^*P_\alpha)\otimes H^0(\G\otimes f^*P_\beta)\otimes f^*P^{-1}_\alpha\otimes f^*P^{-1}_\beta
\end{matrix}.\]
 This map factors through the map
\[\bigoplus_{\gamma\in U+V}H^0(\F\otimes \G\otimes f^*P_\gamma)\otimes f^*P_\gamma^{-1}\rightarrow \F\otimes \G\]
which is, therefore, surjective. Finally, any open subset of $\Pic0 A$ meeting  all the subvarieties $X_i+ Y_j$, for $i=1,\dots ,n$ and $j=1,\dots ,m$, contains some open subset of the form $U+V$, with $U$ and $V$ as  above.
  \end{proof}

  The following result is Theorem \ref{A} of the Introduction, in a slightly more general form. It is an extension of the above quoted result of Debarre asserting that a CGG sheaf with respect to a \emph{finite} onto its image morphism to an abelian variety is ample.

  \begin{theorem}\label{ample-criterion} In the setting of Definition \ref{def:irreg} let us assume that the morphism $f:X\rightarrow A$ is finite onto its image. Let $\F$ be a coherent sheaf on $X$ such that $\F$ is generated by a finite set of subvarieties $\Z=\{Z_i\}_{i\in I}$ strongly spanning  $\Pic0 A$ \emph{(see Subsection \ref{intro-gen})}. Then  $\F$ is ample.
   \end{theorem}
   \begin{proof} To begin with, we note that the fact that the set
    $\Z=\{Z_i\}_{i\in I}$ strongly spans  $\Pic0 A$ implies that there is a positive integer $M$ such that $Z_{i_1}+\dots +Z_{i_M}=\Pic0 A$ for all $(i_1,\dots ,i_M)\in I^M$. 
   By  Lemma \ref{lem:sum} we know that the sheaf $\F^{\otimes M}$ is generated by the set $\{Z_{i_1}+\dots +Z_{i_M}\}_{(i_1,\dots, i_M)\in I^M}$. Hence $\F^{\otimes M}$ is CGG and therefore the same holds true for the symmetric power $S^M\F$. Therefore, by Debarre's theorem, $S^M\F$ is ample, hence $\F$ is ample (\cite[Proposition 2]{kubota}, see also \cite[Theorem 6.1.15]{laz2}).
   \end{proof}
  As a particular case, a coherent sheaf on an abelian variety which is generated by  a single irreducible subvariety $Z$ spanning $\Pic0 A$ is ample.  Obviously this is in general false if $\F$ is generated by  the set of components of a reducible subvariety spanning $\Pic0 A$, but some individual component do not (e.g. on a product of abelian varieties $A=A_1\times A_2$, let $\F= p_1^*\F_1\oplus p_2^*\F_2$, with $\F_i$  CGG (hence ample) sheaves on $A_i$ for $i=1,2$. The sheaf $\F$ is not ample, but is generated by the set $\{\Pic0 A_1\times\{\hat 0\},\{\hat 0\}\times\Pic0 A_2\}$.)

%%%%%%%

\subsection{Theorem \ref{B}} As mentioned in Subsection \ref{intro-CGG} (item {\bf (ii)}), the CGG condition provides a useful criterion for global generation. The natural extension of this is the following\footnote{also this result is an extension of a result of Popa and the author in the context of the above mentioned notion of weak global generation (\cite[Proposition 2.4(c)]{reg2})}

  \begin{proposition}\label{base-locus} In the setting of Definition \ref{def:irreg}, let $\F$ be a coherent sheaf on $X$, generated  (with respect to the morphism $f$) by a set  $\Z=\{Z_i\}$. Let $Y$ be a subvariety of $X$ and let $L$ be a CGG (with respect to $f$)  line bundle on $Y$. Then 
  \[B(\F\otimes L)\subset\bigcup_i\>\Bigl(\,\bigcap_{\alpha\in Z_i}B(L\otimes P_{\alpha}^{-1})\Bigr).
  \]
  \end{proposition}
  \begin{proof} Note that for a line bundle $L$, to be CGG means that the intersection of the base loci $\bigcap_{\alpha\in V}B(L\otimes P_\alpha)$ is empty for all (non-empty) open subsets $V\subseteq  \Pic0 A$. Let $y\in Y$.  Since $L$ is CGG the subset $\{\alpha\in\Pic0 A\>|\> y\not\in B(L\otimes P_\alpha^{-1})\}$ contains an open subset $U_y(L)\subseteq \Pic0 A$. Assume that $y\not\in \bigcup_i(\bigcap_{\alpha\in Z_i}B(L\otimes  P_\alpha^{-1}))$. Then the open set $U_y(L)$ meets all irreducible subvarieties $Z_i$. Therefore, given another open set $U\subset \Pic0 A$ meeting all irreducible subvarieties  $Z_i$, also the open subset $U\cap U_y(L)$ meets all subvarieties $Z_i$. Therefore, in the commutative diagram
  \[\xymatrix{
  \bigoplus_{\alpha\in U\cap U_y(L)}H^0(\F\otimes P_\alpha)\otimes H^0(L\otimes P_\alpha^{-1})\ar[r]\ar[d]
  &
  H^0(\F\otimes L)\ar[d]\\
  \bigoplus_{\alpha\in U\cap U_y(L)}H^0(\F\otimes P_\alpha)\otimes (L\otimes P_\alpha^{-1})_{|y}\ar[r]
  &
  (\F\otimes L)_{|y}
  }
  \]
 both the left  arrow and the bottom arrow are surjective. Hence the right arrow is surjective.
   \end{proof}
  
  %%%%%%

  \section{Relationship with the FMP transform}\label{FM}
  
 \subsection{The basic relation} The relevance of the FMP functor  in the study of the generation of coherent sheaves on abelian varieties stems from the following  relation between the evaluation maps of a coherent sheaf on $A$  and of its naive FMP transform on $\Pic0 A$. This is known to the experts (see  Schnell's paper \cite{schnell}, proof of Proposition 4.1). In what follows we will keep the setting and notation of the Introduction, Subsection \ref{relation}. 
  
  Given an open subset $U$ of $\Pic0 A$, the continuous evaluation map at $x\in A$
\begin{equation}\label{twisted}
{\rm ev}_U(x):\bigoplus_{\alpha\in U}H^0(\F\otimes P_\alpha)\otimes P_\alpha^{-1}\rightarrow \F_{|x}
\end{equation}
 factors through the map
\begin{equation}\label{twisted2}
\bigoplus_{\alpha\in U}H^0(\F\otimes P_\alpha)\otimes (P_\alpha^{-1})_{|x}\rightarrow \F_{|x}
\end{equation}
and (\ref{twisted}) is surjective if and only if (\ref{twisted2}) is. 

In order to render more transparent some steps of the argument, we will occasionally skip the notation $\F_{|x}:=\F\otimes k(x)$. Let us consider the individual maps of (\ref{twisted2}), i.e.
\begin{equation}\label{twisted2-ind}
H^0(\F\otimes P_\alpha)\otimes P_\alpha^{-1}\otimes k(x)\rightarrow \F\otimes k(x).
\end{equation}
Let us recall that, via the canonical isomorphism between the abelian variety $A$ and $\Pic0 (\Pic0 A)$, a point $x\in A$ corresponds to the line bundle  on $\Pic0 A$:
\[
P_x:=\cP_{|\{x\}\times\Pic0 A}.
\]

The next Proposition describes the image of the dual map of (\ref{twisted2-ind}) via the (contravariant) equivalence
\[F: D(A)\rightarrow D(\Pic0 A), \qquad F(\>\cdot\>)=\Phi^A_{\cP^{-1}}((\>\cdot\>)^\vee)
\]

\begin{proposition}\label{identi1} The functor $F$ identifies the Serre-Grothendieck dual of the linear map (\ref{twisted2-ind}) to the linear  map
\[H^g(\Phi^A_{\cP^{-1}}(\F^\vee)\otimes P_x)\rightarrow (R^g\Phi^A_{\cP^{-1}}(\F^{\vee})\otimes  P_x)\otimes k(\alpha)
\]
factoring as follows
\begin{equation}\label{factorization}
\xymatrix{
H^g(\Phi_{\cP^{-1}}(\F^\vee)\otimes P_x)\ar[r]^{\mathrm{ed}_x}\ar[rd]&H^0(R^g\Phi_{\cP^{-1}}(\F^\vee)\otimes P_x)\ar[d]^{\mathrm{ev}_x(\alpha)}\\ 
& (R^g\Phi_{\cP^{-1}}(\F^{\vee})\otimes  P_x)\otimes k(\alpha)}
\end{equation}
where $ev_x(\alpha)$ is the evaluation at the point $\alpha\in\widehat A$ of the coherent sheaf $R^g\Phi_{\cP^{-1}}(\F^\vee)\otimes P_x$. 
\end{proposition}

\begin{proof}The proof is straightforward. We identify the map (\ref{twisted2-ind}) to
\begin{equation}\label{twisted3}
H^0(\F\otimes P_\alpha)\otimes H^0(P_\alpha^{-1}\otimes k(x))\rightarrow  H^0(\F\otimes k(x)).
\end{equation}
Applying Serre-Grothedieck duality, we write  the dual map of (\ref{twisted3}) as
\begin{equation}\label{twisted4}\mathrm{Hom}_{D(A)}(k(x),\F^\vee[g])=\mathrm{Ext}^g(k(x),\F^\vee)\rightarrow \mathrm{Hom}(H^0(P_\alpha^{-1}\otimes k(x)), H^g(\F^\vee\otimes P_\alpha^{-1})).
\end{equation}
Applying the functor $\Phi_{\cP^{-1}}^A$ to the the source of (\ref{twisted4}),   we have the chain of isomorphisms
\begin{equation}\label{chain}\begin{aligned}\mathrm{Hom}_{D(A)}(k(x),\F^\vee[g])&\cong \mathrm{Hom}_{D(\Pic0 A)}(P_x^{-1}, \Phi^A_{\cP^{-1}}(\F^\vee)[g])\cong \\
 \cong\mathrm{Ext}^g(P_x^{-1}, \Phi^A_{\cP^{-1}}(\F^\vee))&\cong H^g(\Phi^A_{\cP^{-1}}(\F^\vee)\otimes P_x).
\end{aligned}
\end{equation}
Concerning the target of the map (\ref{twisted4}), there are the canonical identifications
\[H^0(P_\alpha^{-1}\otimes k(x))\cong R^0\Phi^A_{\cP^{-1}}\bigl(k(x)\bigr)\otimes k(\alpha)=\Phi^A_{\cP^{-1}}(k(x))\otimes k(\alpha)=P_x^{-1}\otimes k(\alpha)
\]
and 
\[H^g(\F^\vee\otimes P_\alpha^{-1})\cong \bigl(R^g\Phi^A_{\cP^{-1}}(\F^{\vee})\bigr)\otimes k(\alpha).
\]

We conclude that the map (\ref{twisted4}) is identified, via the functor $\Phi^A_{\cP^{-1}}$, to a linear map
\[H^g(\Phi^A_{\cP^{-1}}(\F^\vee)\otimes P_x)\rightarrow \bigl(R^g\Phi^A_{\cP^{-1}}(\F^{\vee})\otimes  P_x\bigr)\otimes k(\alpha).
\]
factorizing trough the evaluation map of the sheaf  $R^g\Phi^A_{\cP^{-1}}(\F^{\vee})\otimes  P_x$ at the point $\alpha$, as stated in (\ref{factorization}). \end{proof}

\begin{remark} The map
 \begin{equation}\label{edge} 
 {\mathrm{ed}_x}:{H}^g(\Phi_{\cP^{-1}}^A(\F^\vee)\otimes P_x)\rightarrow H^0(R^g\Phi_{\cP^{-1}}^A(\F^\vee)\otimes P_x).
 \end{equation}
 appearing in (\ref{factorization}) 
 is in fact the edge map in  the hypercohomology spectral sequence
 \[H^p(R^q\Phi^A_{\cP^{-1}}(\F^\vee)\otimes P_x)\Longrightarrow {H}^{p+q}(\Phi^A_{\cP^{-1}}(\F^\vee)\otimes P_x).
\]
\end{remark}

In the next section we will use also the following version of Proposition \ref{identi1}, proved exactly in the same way. This time, for $x\in A$ and $\alpha\in \Pic0 A$, we consider the map
\begin{equation}\label{twisted-inv}
H^0(\F\otimes P_\alpha)\otimes  k(x)\rightarrow \F\otimes P_\alpha\otimes k(x).
\end{equation}

\begin{proposition}\label{identi11} The functor $F$ identifies the Serre-Grothendieck dual of the linear map (\ref{twisted-inv}) to the linear  map
\[H^g(\Phi^A_{\cP^{-1}}(\F^\vee)\otimes P_x)\otimes P_x^{-1}\otimes k(\alpha)\rightarrow (R^g\Phi^A_{\cP^{-1}}(\F^{\vee}))\otimes k(\alpha)
\]
factoring as follows
\begin{equation}\label{factorization-bis}
\xymatrix{
H^g(\Phi_{\cP^{-1}}(\F^\vee)\otimes P_x)\otimes P_x^{-1}\otimes k(\alpha)\ar[r]\ar[rd]&H^0(R^g\Phi_{\cP^{-1}}(\F^\vee)\otimes P_x)\otimes P_x^{-1}\otimes k(\alpha)\ar[d]\\ 
& (R^g\Phi_{\cP^{-1}}(\F^{\vee}))\otimes k(\alpha).}
\end{equation}
\end{proposition}
\begin{proof}
The proof is similar to the previous one, we just point out what is different. We write the map (\ref{twisted-inv}) as
\[H^0(\F\otimes P_\alpha)\otimes H^0( k(x))\rightarrow H^0( \F\otimes P_\alpha\otimes k(x))
\]
Dualizing we get the map
\[\mathrm{Ext}^g(k(x)\otimes P_\alpha, \F^\vee)\cong \mathrm{Ext}^g(k(x),\F^\vee)\otimes (P_\alpha^\vee\otimes k(x))\rightarrow \mathrm{Hom}(H^0(k(x)), H^g(\F^\vee))\otimes (P_\alpha^\vee\otimes k(x))\]
Applying the functor $\Phi^A_{\cP^{-1}}$ to the source we get, after a small calculation,  $H^g(\Phi_{\cP^{-1}}(\F^\vee)\otimes P_x)\otimes P_x^{-1}\otimes k(\alpha)$.  The target is identified to $(R^g\Phi_{\cP^{-1}}(\F^{\vee})\otimes P_x)\otimes P_x^{-1}\otimes k(\alpha)\cong(R^g\Phi_{\cP^{-1}}(\F^{\vee}))\otimes k(\alpha)$.
\end{proof} 

\subsection{Consequences } From this point we will adopt the notation of the Introduction  (see \ref{T(F)}), namely
\begin{equation}\label{T(F)2}
\T(F):=R^g\Phi^A_{\cP^{-1}}(\F^\vee).
\end{equation} This sheaf will be referred to as the \emph{naive FMP transform of $\F$}. 
Proposition \ref{identi1} allows to express the surjectivity of the map ${\rm ev}_U(x)$ of (\ref{twisted}) as follows (this is Theorem \ref{C}(a) of the Introduction). 

\begin{corollary}\label{gen-FM} Let $U$ be an open subset of $\Pic0 A$ and let $x\in A$. The map ${\rm ev}_U(x)$ of
(\ref{twisted})
is surjective if and only if the following two conditions hold:

\noindent (1) the map ${\rm ed}_x$ of (\ref{edge}) is injective;

\noindent (2) the restriction of the simultaneous evaluation map 
\begin{equation}\label{simultaneous}
{\rm ev}_x(U):H^0(\T(\F)\otimes  P_x)\longrightarrow\prod_{\alpha\in U}(\T(\F)\otimes  P_x)_{|\alpha}
\end{equation}
to the image of the map ${\rm ed}_x$ 
is injective.
\end{corollary}

\begin{remark}  Notice that the map ${\rm ed}_x$ depends only on $x$ and not on the open subset $U\subset \Pic0 A$. Therefore if $\F$ is generated at $x$, i.e. the map ${\rm ev}_U(x)$ (see (\ref{twisted})) is surjective for \emph{some} $U$, then ${\rm ed}_x$ is injective.
\end{remark}

\begin{remark}\label{kernel} In view of the previous corollary, it is useful to understand the kernel of the evaluation maps $\mathrm{ev}_x(U)$ of (\ref{simultaneous}) for a non empty open subset $U$. A non-zero section $s\in \ker ( ev_x(U))$ must be a \emph{torsion section}, i.e. the image of the associated map $s:\OO_{\Pic0 A}\rightarrow \T(\F)\otimes P_x$ must be a torsion sheaf. If \emph{the support of such image is reduced} then it is the closure of the subset of $\alpha\in\Pic0 A$ such that $s_{|\alpha}\in (\T(\F)\otimes  P_x)_{|\alpha}$ is non zero. If this is the case then $U$ must be contained in the complement of such support. 
Of course for all $x\in A$ the torsion subsheaves of $\T(\F)\otimes P_x$ coincide, after tensorization with $P_x^{-1}$, with the torsion subsheaves of $\T(\F)$, hence they do not depend on $x\in A$.  Moreover all the irreducible components of supports of torsion subsheaves of $\T(\F)$ appear in the support of sheaves appearing in the torsion filtration of $\T(\F)$  (\cite[Definition 1.1.4]{huybrechts-lehn}) . 
\end{remark}

Summarizing, we have the following Corollary, which  is Theorem \ref{C}(b) of the Introduction. In the statement we denote
\[\widetilde{V}^0(\tau)=\{x\in A\>|\> H^0(\tau\otimes P_x)\cap {\rm Im}({\rm ed}_x)\ne 0\}
\]

\begin{corollary}\label{generated} Let $\F$ be a coherent sheaf on an abelian variety $A$. Assume that:

\noindent (a)  the maps ${\rm ed}_x$ of (\ref{edge}) are injective for all $x\in A$;\\
(b) all sheaves $\tau$  appearing in the torsion filtration of the naive FMP transform $\T(\F)$, and such that $\widetilde{V}^0(\tau)$ is non empty, have   reduced scheme-theoretic support. 

\noindent Then  $\F$ is generated by  the set of minimal irreducible components of supports of all such sheaves~$\tau$.
\end{corollary}

\begin{remark}\label{GV-torsion} In general, it is not known how to identify in function of $\F$ the various sheaves $\tau$ appearing in the torsion filtration of the naive FMP transform $\T(\F)$, let alone those such that $\widetilde{V}_0(\tau)$ in non empty. Not surprisingly, in the case of GV sheaves this can be done. Recalling the definition and notation for cohomological support loci given in (\ref{Vi}), 
 it is well known that $F$ is a GV sheaf if and only if, for all $i\ge 0$
 \[\codim_{\Pic0 A}V^i(\F)\ge i
 \]
(\cite[Corollary 3.10]{GV} or \cite[Theorem 2.3]{reg3}).  If this is the case, the  components of supports of the torsion sheaves appearing in the torsion filtration of the sheaf $\T(\F)$ are components $W$ of the loci $V^i(\F)$ of the minimal codimension, namely $\codim_{\Pic0 A}W= i$ (this follows from a result of Popa and the author, see \cite[Theorem 1.10]{msri}). We refer  also to \cite[\S 9]{desperados} where an  even more precise description of the torsion filtration is given in the case of \emph{coherent sheaves admitting a Chen-Jiang decomposition}, a special class of GV-sheaves which includes direct images of pluricanonical bundles with respect to morphism to an abelian variety. See  Remark \ref{CJ} below for more on these sheaves.
\end{remark}

%%%%%%%%%%%%

\section{Generation and ampleness of naive FMP transforms and generalized Picard bundles. }\label{naive-picard}

A known class of examples of ample vector bundles on abelian varieties is the one of (dual) \emph{Picard bundles} (see Subsection \ref{pic}). 
In the case where $X$ is a smooth curve naturally embedded in its Jacobian, it was classically known that the projectivization of the dual of the Picard bundle of a line bundle of degree $d\ge 2g-1$ is the $d$-symmetric product of the curve. Based on this observation, it follows that the dual of Picard bundles of curves are ample, and, as a consequence, one gets the existence and connectedness theorems of Brill-Noether theory (\cite[Chapter VII]{acgh}).

 The ampleness of dual Picard bundles and some of its applications were subsequently generalized to all smooth complex projective varieties by Lazarsfeld (\cite[\S 6.3.C and 7.2.C]{laz2}). 
In this section we show that the FMP methods of the previous section quickly provide (with a different argument) a vast generalization of Lazarsfeld's result.

\subsection{Naive FMP transforms of sheaves on abelian varieties} In order to give a quick idea of the result, we start from a particular case already met in the previous section, namely naive FMP transforms of coherent sheaves on abelian varieties. We keep the notation of Subsection \ref{pic}, especially on the set of subvarieties $\Z(\F)$ (irreducible components of supports of coherent sheaves appearing in the torsion filtration of $\F$).

\begin{proposition}\label{naiveFMP} Let $\F$ be a coherent sheaf on an abelian variety $A$ such that all of its subsheaves have reduced scheme-theoretic support. Then the naive FMP transform $\T(\F)=R^g\Phi^A_{\cP^{-1}}(\F^\vee)$  is generated by a subset of the set $\Z(\F)$ \emph{(see Theorem \ref{D})}. In particular, $\T(\F)$ is an ample sheaf on $\Pic0 A$ as soon as all subvarieties  appearing in the set $\Z(\F)$ span $A$.
\end{proposition}

As  mentioned in the Introduction, this result, for torsion free coherent sheaves, is already present in Schnell's paper \cite[Theorem 4.1]{schnell}. The present argument is borrowed from his. 

\begin{proof}   We will use the notation  (\ref{twisted}) and (\ref{simultaneous}) on evaluation maps (keeping in mind that, since now we are dealing with the generation of $\T(\F)$, which is a sheaf on the abelian variety $\Pic0 A$, the role of $A$ and $\Pic0 A$ is exchanged with respect to the previous section). Specifically, the generation of $\T(\F)$ means the surjectivity of the map
\[ ev_A(\alpha): \bigoplus_{x\in A} :H^0(R^g\Phi^A_{\cP^{-1}}(\F^\vee)\otimes P_x)\otimes (P_x^{-1})_{|\alpha}\longrightarrow R^g\Phi^A_{\cP^{-1}}(\F^\vee)_{|\alpha}
\]
for all $\alpha\in\Pic0 A$.  
The argument consists in considering diagram (\ref{factorization-bis}) and  let $x$ (rather than $\alpha$) vary.  We recall that, by Proposition \ref{identi11},  the evaluation map
\[
H^0(\F\otimes P_\alpha)\rightarrow (\F\otimes P_\alpha)_{|x}
\]
is the dual of a map factoring through the twisted evaluation map
\[
H^0(R^g\Phi^A_{\cP^{-1}}(\F^\vee)\otimes P_x)\otimes (P_x^{-1})_{|\alpha}\longrightarrow R^g\Phi^A_{\cP^{-1}}(\F^\vee)_{|\alpha}.
\]
as in diagram (\ref{factorization-bis}). 
In the same way, the map
\[ev_{\alpha}(A): H^0(\F\otimes P_\alpha)\longrightarrow \prod_{x\in A}(\F\otimes P_\alpha)_{|x}
\]
is the dual of a map factoring trough $ev_A(\alpha)$. Therefore the surjectivity  the map $ev_A(\alpha)$ follows from the injectivity of the map $ev_\alpha(A)$, which holds because we are assuming there are no subsheaves of $\F$ with non-reduced scheme-theoretic support (hence for each section $s\in H^0(\F\otimes P_\alpha)$ there is a point $x\in A$ such that $s(x)\ne 0$). By the same reason  also the map
\[ev_{\alpha}(V): H^0(\F\otimes P_\alpha)\longrightarrow \prod_{x\in V}(\F\otimes P_\alpha)_{|x}
\]
is injective, for all open subsets $V$ of $A$ meeting all irreducible components of the support of all subsheaves of $\F$. In fact it is enough to consider the subsheaves appearing in the torsion filtration of $\F$. 
\end{proof}

\subsection{Generalization to transforms from reduced schemes mapping to abelian varieties } In this subsection we prove  Theorem \ref{D} of the Introduction, namely the generalization of Proposition \ref{naiveFMP} to naive transforms from certain reduced schemes mapping to abelian varieties, rather than from abelian varieties.  Again, we keep the notation and setting introduced in Subsection \ref{pic}.

\begin{theorem}\label{Picard}  Let $f:X\rightarrow A$ be a reduced Cohen-Macaulay equidimensional (of dimension $d$) projective scheme equipped with a morphism to an abelian variety. Let $\F$ be a coherent sheaf on $X$ such that all of its subsheaves have reduced scheme theoretic support. Then its naive FMP transform $\T(\F)=R^d\Phi^X_{\cP_X^{-1}}(\Delta_X(F))$  is generated by a subset of the set $f(\Z(\F))$. In particular, $\T(\F)$ is ample sheaf on $\Pic0 A$ as soon as, for each $i$, the image $f(Z_i)$ of the subvariety $Z_i\in \Z(\F)$ spans $A$ \emph{(e.g. this happens as soon as $f(X)$ spans $A$ and $\F$  is a torsion free sheaf on $X$.)}
\end{theorem}

\begin{proof} The argument is similar to the one of Proposition \ref{naiveFMP} but needs to be adjusted  because the functor $\Phi^X_{\cP_X^{-1}}$ in general is not an equivalence, hence Proposition \ref{identi1} and Proposition \ref{identi11} do not hold anymore. Even if we will need only an adjustment of Proposition \ref{identi11}, in order to follow the argument given in the previous section, we start from Proposition \ref{identi1}. We still consider, for $x\in X$ and $\alpha\in \Pic0 A$, the linear map
\[H^0(\F\otimes f^*P_\alpha)\otimes H^0(f^*P_\alpha^{-1}\otimes k(x))\rightarrow  H^0(\F\otimes k(x))
\]
Dualizing we get
\begin{equation}\label{irreg}
\mathrm{Hom}_{D(X)}(k(x),\Delta_X(\F)[d])=\mathrm{Ext}^d(k(x),\Delta_X(\F))\rightarrow \mathrm{Hom}(H^0(f^*P_\alpha^{-1}\otimes k(x)), H^d(\Delta_X(\F)\otimes f^*P_\alpha^{-1})).
\end{equation}
We still have that $\Phi^X_{\cP_X^{-1}}(k(x))=R^0\Phi^X_{\cP_X^{-1}}(k(x))=P^{-1}_{f(x)}$ and 
\[H^0(f^*P_\alpha^{-1}\otimes k(x))\cong R^0\Phi_{\cP_X^{-1}}(k(x))\otimes k(\alpha)=\Phi_{\cP_X^{-1}}(k(x))\otimes k(\alpha)=
P_{f(x)}^{-1}\otimes k(\alpha)
\]
Therefore the target of the map (\ref{irreg}) is still identified by the functor $\Phi^X_{\cP_X^{-1}}$ to the fibre
\[(R^d\Phi_{\cP_X^{-1}}(\Delta_X(F))\otimes f^*P_{f(x)})_{|\alpha}.
\]
Concerning the source of the map (\ref{irreg}), we note that, although the analog, in the present setting, of the first map in the chain of isomorphisms (\ref{chain}) is not an isomorphism anymore, nevertheless it follows
that the map (\ref{irreg}) factors, via the functor $\Phi^X_{\cP^{-1}_X}$,  trough the evaluation map of global sections at the point $\alpha\in\Pic0 A$ of the naive FMP transform, twisted by the line bundle $P_{f(x)}$:
\begin{equation}\label{factorization2}
\xymatrix{
\mathrm{Ext}^d(k(x),\Delta_X(\F))\ar[r]\ar[rd]_{(\ref{irreg})}&H^0(R^d\Phi_{\cP_X^{-1}}(\Delta_X(F))\otimes f^*P_{f(x)})\ar[d]^{\mathrm{ev}_x(\alpha)}\\ 
& (R^d\Phi_{\cP_X^{-1}}(\Delta_X(F))\otimes f^*P_{f(x)})\otimes k(\alpha)}.
\end{equation}
Turning to the adjustment of Proposition \ref{identi11} to the present setting, we consider, for $x\in X$ and $\alpha\in\Pic0 A$, the linear map
\begin{equation}\label{eq:added} H^0(\F\otimes f^*P_\alpha)\otimes k(x)\rightarrow \F\otimes f^*P_\alpha\otimes k(x)
\end{equation}
In a completely similar way it follows that the dual of the map (\ref{eq:added}) factors,  via the functor $\Phi^X_{\cP^{-1}_X}$, through the twisted evaluation map at the point $\alpha\in\Pic0 A$ of the naive FMP transform:
\begin{equation}\label{factorization-bis2}
\xymatrix{
\mathrm{Ext}^d(k(x)\otimes f^*P_\alpha,\Delta_X(\F))\ar[r]\ar[rd]&H^0(R^d\Phi_{\cP_X^{-1}}(\Delta_X(F))\otimes f^*P_{f(x)})\otimes f^*P_{f(x)}^{-1}\otimes k(\alpha)\ar[d]^{\mathrm{ev}_x(\alpha)}\\ 
& (R^d\Phi_{\cP_X^{-1}}(\Delta_X(F))\otimes k(\alpha)}
\end{equation}
At this point the argument goes   exactly as in Proposition \ref{naiveFMP}.
\end{proof}

\begin{example}\label{loc-free} Note that in  Proposition \ref{naiveFMP} and  Theorem \ref{Picard}  we are not assuming the vanishing condition (\ref{IT(0)Pic}), namely that $H^i(X,\F\otimes f^*P_\alpha)=0$ for all $\alpha\in\Pic0 A$ and $i>0$, which implies that the naive FMP transform (or Picard sheaf) $\T(\F)$ is  a locally free sheaf on $\Pic0 A$. It is worth to note that in fact there are many examples where such condition does not hold but still $\T(\F)$ is locally free. Therefore, even if one is interested only in ample \emph{vector bundles} the above results produce a wider class of examples.   

 Let us outline a simple way to produce such sheaves,   similar to Raynaud's construction of stable vector bundles  "without theta divisor" on curves (\cite{raynaud}). The simplest example is as follows. Let $C$ be a smooth curve embedded in its Jacobian $J(C):=A$. For odd coprime positive integers $a,b$ let $W_{a,b}$ be the semhomogenous vector bundles on a p.p.a.v., as $J(C)$, considered by Mukai and Oprea (\cite[Theorem 7.1 and Remark 7.3]{semihom}, \cite{oprea}). Denoting $\mu_a:A\rightarrow A$ the multiplication by $a$, and $L=\OO_A(\Theta)$, these are vector bundles on $A$  such that 
 \begin{equation}\label{mu}
 \mu_a^*W_{a,b}=(L^{ab})^{\oplus a^g}.
 \end{equation}
  We claim that if $a$ and $b$ are such that $\frac{b}{a}\in(0,1)$ then $h^i((W_{a,b})_{|C}\otimes P_\alpha)$ is non-zero and constant for $\alpha\in \Pic0 A$ for both for $i=0,1$. This can be shown as follows. By (\ref{mu}),  we have that
  \[h^i((W_{a,b})_{|C}\otimes P_\alpha)=\frac{1}{a^{g}}h^i(\mu_a^*(\OO_C\otimes P_\alpha)\otimes \OO_A(ab\Theta)).
  \]
   for a general $\alpha\in\Pic0 A$. This shows that, for $i=0,1$ and $\alpha$ general in $\Pic0 A$, by definition, the rational number $\frac{1}{a^g}h^i((W_{a,b})_{|C}\otimes P_\alpha)$ equals the value of the \emph{cohomological rank functions} $h^i_{\OO_C}(x\utheta)$, $x=\frac{b}{a}$ (see \cite{ens}). These functions have been computed in \emph{loc cit}, p. 837, proof of Theorem 7.5. Specifically, in the interval $[0,1]$, they are
   \begin{equation}\label{AJcurve} h^0_{\OO_C}(x\utheta)=x^g\qquad\hbox{and}\qquad h^1_{\OO_C}(x\utheta)=xg-g+1-x^g.
   \end{equation}
    If there was a non empty jump locus for the function $\alpha\mapsto h^i((W_{a,b})_{|C}\otimes P_\alpha)$ ($\alpha\in \Pic0 A$) then as it is easy to see, this would cause a critical point of the functions $h^i_{\OO_C}(x\utheta)$ (see \S5, \emph{loc cit}) in the interval $(0,1)$ but (\ref{AJcurve}) shows, in particular, that there is none. Therefore the functions $\alpha\mapsto h^i((W_{a,b})_{|C}\otimes P_\alpha)$ are constant. It follows, in particular, that $\T((W_{a,b})_{|C})$ is a locally free sheaf on $A$.
   
   Similar examples can be obtained starting from any subvariety $X$ of a ppav and any coherent sheaf $\G$ on $X$, considering the sheaves $\F=\G\otimes W_{a,b}$ for $\frac{b}{a}$ in a suitable range. Constructions of this type can be performed on abelian varieties with arbitrary polarizations (not necessarily principal). 
\end{example}

%%%%%%%%%%%%%

\section{Applications to Brill-Noether theory }\label{sub:BN} 

\subsection{Singular curves equipped with a morphism to an abelian variety} The interest of results as Theorem \ref{Picard}, besides providing examples of ample vector bundles, is in their application to Brill-Noether theory via  the nowdays classical results of Kempf, Kleiman-Laksov,  and Fulton-Lazarsfeld on non-emptyness and connectedness of degeneracy loci \cite[Chapter VII]{acgh}, \cite[\S 7.2.C]{laz2}. Such applications mainly concern locally free sheaves on smooth curves. Via Theorem \ref{D} they can be generalized e.g. to the following setting: \emph{torsion free sheaves $\F$ on singular curves  $X$  (connected reduced 1-dimensional Cohen-Macaulay projective schemes), equipped with a morphism $f:X\rightarrow A$ to an abelian variety, such that the image of each component of spans the abelian variety $A$.} 

In this sort of application we will be mainly concerned with Brill-Noether  loci
\[V_f^{0, r}(X,\F)=\{\alpha\in \Pic0 A\>|\> h^0(X,\F\otimes f^*P_\alpha)\ge r+1\}
\]
According to the previous notation, the locus $V_f^{0, 0}(X,\F)$ is simply denoted  $V_f^{0}(X,\F)$. The loci $V_f^{0, r}(X,\F)$ can be realized as degeneracy loci of a map of locally free sheaves in the usual way, namely considering the exact sequence
\[0\rightarrow  \F\rightarrow \F(D)\rightarrow \F(D)_{|D}\rightarrow 0
\]
where $D$ is a Cartier divisor avoiding the singular points of $X$,  such that $D$ is of degree high enough on all components of $X$ so that $H^i(X,\F(D)\otimes f^*P_\alpha)=0$ for all $\alpha\in\Pic0 A$. It turns out  that $\Phi^X_{\cP_X}(\F(D))$ coincides with $R^0\Phi^X_{\cP_X}(\F(D))$, and  is a locally free sheaf on $\Pic0 A$ (dual to the naive FMP transform $\T(\F)$) whose fibre at $\alpha$ is identified to $H^0(X,\F(D)\otimes f^*P_\alpha)$.  Therefore the locus $V_f^{0, r}(X,\F)$ is the $k$'th degeneracy locus of the map of vector bundles
\[\varphi: \Phi^X_{\cP_X^{-1}}(\F(D))\rightarrow   \Phi^X_{\cP_X^{-1}}(\F(D)_{|D})
\]
where $k=\chi(\F(D))-(r+1)$. 
The target is  a locally free sheaf on $\Pic0 A$ of rank  equal to $\sum \nu_i\deg D_i$, where  $(\nu_1,..,\nu_h)$ is the \emph{multirank} of $\F$ (the tuple of the ranks of $\F$ at the various components of $X$) and $\deg D_i$ is the degree of $D$ at the various components of $X$.  The expected codimension of the locus $V_f^{0, r}(X,\F)$ is
\[(r+1)(\sum \nu_i\deg D_i)-(\chi(\F(D))-(r+1)))=(r+1)(r+1-\chi(F))\]
By Theorem \ref{Picard} the source of the map $\varphi$ is negative. On the other hand it is well known that the target is a homogeneous vector bundle, namely admitting a filtration whose quotients are line bundles  parametrized   by $\Pic0 (\Pic0 A))=A$. Therefore, by the Theorem of Fulton-Lazarsfeld, we have

\begin{theorem}\label{BN-sing} In the above setting,
the locus $V_f^{0, r}(X,\F)$ is nonempty (resp. connected) as soon as 
\[(r+1)(r+1-\chi(F))\le \dim A\qquad\hbox{ {\rm (resp.} }\>(r+1)(r+1-\chi(F))<\dim A\hbox{\rm )}. 
\]
In particular $V^0_f(\F)$ is non empty as soon as 
\[\chi(\F)\ge -\dim A+1.
\] 
\end{theorem}
This statement recovers (for smooth curves embedded in their Jacobians) the classical existence and connectedness theorems of Brill-Noether theory, as well as Ghione's theorem (\cite[Example 7.2.13]{laz2}).
As a consequence, we have the following version of the theorem of Segre-Nagata.  To simplfy, assume furthermore that, in the setting of Theorem \ref{Picard}, the torsion free sheaf $\F$ has \emph{uniform} rank  $\nu$ (i.e. $\nu_1=...=\nu_h=\nu$). Let us define
\[\deg \F:=\chi(F)-\nu(1-p_a(X))
\]

\begin{corollary} The sheaf $\F$ has an invertible subsheaf $B$ of degree 
\[\deg B\ge 1-p_a(X)+\frac{{\deg \F}+\dim A-1}{\nu}.
\] 
\end{corollary}
\begin{proof} By the last inequality of Theorem \ref{BN-sing} it follows that the locus $V_f^0(\F\otimes B^{-1})$ is non-empty as soon as $\chi(F)-\nu\deg B\ge -\dim A+1$. 
\end{proof}

Similarly, one can extend to the setting of Theorem \ref{Picard} the result of \cite[Example 7.2.15]{laz2}.

%%%%%%%%%

\subsection{Theorem \ref{cast-de-franchis}}
Let $\F$ be a coherent sheaf on an abelian variety $A$.  Recalling the notation of (\ref{Vi}) and ({\ref{V>0}) about the cohomological support loci of $\F$, assume that the loci $V^{>0}(\F)$ is strictly contained in $\Pic0 A$. Then $\chi(\F)\ge 0$ (since $\chi(\F)$ is equal to the generic value of $h^0(\F\otimes P_\alpha)$). It turns out that this easy inequality can be improved if the cohomological loci $V^i(\F)$, $i>0$, are suitably small. For example, the "Castelnuovo - De Franchis inequality" of Popa and the author (\cite[Theorem 3.3]{duke}) states that,  if $\F$ is a GV sheaf and $V^{>0}(\F)$ is non empty, then, for all $\alpha\in\Pic0 A$, 
\begin{equation}\label{gv-index}
\chi(\F)\ge \hbox{min}_{i>0}{\rm codim}_{\alpha}\{V^i(\F)- i\>|\>i>0\}
\end{equation}
for all $\alpha\in\Pic0 A$, where ${\rm codim}_\alpha V^i(\F)$ denotes the codimension of $V^i(\F)$ in $\Pic0 A$, in the neighborhood of a given point $\alpha\in\Pic0 A$ (the right hand side of (\ref{gv-index}) is called \emph{generic vanishing index at the point $\alpha$}). 

The result stated as Theorem \ref{cast-de-franchis} in the Introduction complements the inequality (\ref{gv-index}), dealing  with the case where the locus $V^{>0}(\F)$ is empty. The statement asserts that if this is the case, and  the coherent sheaf $\F$ is torsion free, or, more generally, all of its subsheaves  have all reduced support, and each irreducible component of the support spans $A$, then
\begin{equation}\label{chi}
\chi(\F)\ge {\rm hd}(\F)+1.
\end{equation}
\begin{proof} 
(of Theorem \ref{cast-de-franchis}) By Theorem \ref{D} the naive FMP transform of $\F$, namely
\[\T(\F)=R^g\Phi_{\cP^{-1}}^A(\F^\vee)
\]
is an ample sheaf on $\Pic0 A$. Moreover the hypothesis that $V^{>0}(\F)=\emptyset$ yields, by base change, that $\T(\F)$ is a locally free sheaf on $\Pic0 A$.

We claim that
\begin{equation}\label{claim} {\rm Supp}\,\bigl(\mathcal Ext^i(\F,\OO_A)\bigr)\subseteq V^i(\T(\F))
\end{equation}
The assertion of Theorem \ref{cast-de-franchis} follows  from the claim. Indeed, by Le Potier vanishing, $V^j(\T(\F))=\emptyset$ as soon as $j\ge {\rm rk}\T(\F)=\chi(\F)$. Hence (\ref{claim}) yields that $\chi(\F)>{\rm max}\,\{j\>\>|\>\> \mathcal Ext^j(\F,\OO_A)\ne 0\}=\mathrm{hd}(\F)$.

To prove (\ref{claim}), note that the hypothesis $V^{>0}(\F)=\emptyset$ yields trivially that  $R^i\Phi_{\cP^{-1}}^A(\F^\vee)=0$ for $i\ne g$, and therefore
 $\T(\F)=\Phi_{\cP^{-1}}^A(\F^\vee)[g]$ (in other words  $\F$ is a GV-sheaf (see (\ref{GV-sheaf})). 
 Therefore, since the inverse of the FMP equivalence $ \Phi_{\cP^{-1}}^A:D(A)\rightarrow D(\Pic0 A)$ is the equivalence $\Phi_\cP^{\Pic0 A}[g]:D(\Pic0 A)\rightarrow D(A)$, it follows that
 \[\Phi_\cP(\T(\F))=\F^\vee={\bf R}\mathcal Hom(\F,\OO_A)\]
 hence 
 \[R^i\Phi_\cP(\T(\F))=\mathcal Ext^i(\F,\OO_A)\]
 (see also Proposition \ref{T(T)} below for a related statement). This yields (\ref{claim}) since, by base change,
 \[{\rm Supp}\,R^i\Phi_\cP(\T(\F))\subseteq V^i(\T(\F)).
 \]
\end{proof}
\begin{remark} Interestingly, the  main point of the proof of Theorem \ref{cast-de-franchis} is Le Potier vanishing theorem,  while the essential ingredient of the proof of inequality (\ref{gv-index}) is the syzygy theorem of Evans-Griffith  (\cite[Corollary 1.7]{evans-griffith}), which is in turn  quite related to the theorem of Le Potier (as shown by Ein, \cite{ein}). Therefore it seems possible that the inequalities (\ref{gv-index})  and (\ref{chi}) are both manifestations of some more general phenomenon.  
\end{remark}

Theorem \ref{cast-de-franchis} was  already implicitly used by the author (in a special case) in the proof of the following result (conjectured in \cite[\S6]{debarre-hacon}). 

\begin{corollary}[\cite{sing} Theorem B(1)] \label{nadel} Let $A$ be a complex simple  abelian variety and $D$ an effective  $\mathbb Q$-divisor on $A$ such that  $(A,D)$ is a non klt pair. If $L$ is a line bundle on $A$ such that $L-D$ is nef and big then
\[\chi(L)\ge \dim A+1.
\] 
\end{corollary}
\begin{proof} The hypothesis means that $\J(D)$, the multiplier ideal sheaf of $D$, is non trivial. By Nadel's vanishing $V^{>0}(\J(D)\otimes L)=\emptyset$. Let $Z$ be the zero-scheme of $J(D)$. By Theorem \ref{cast-de-franchis}, $\chi(\J(D)\otimes L)\ge {\rm codim}\,Z$, where ${\rm codim\, Z}$ (resp. $\dim Z$) denotes the \emph{maximal codimension} (resp. \emph{maximal dimension}) of a component of $Z$. On the other hand, an inequality of Debarre-Hacon (\cite[Lemma 5(e)]{debarre-hacon}) states that, if $A$ is simple, $\chi(\J(D)\otimes L)\le \chi(L)-\dim Z-1$. Combining the two inequalities it follows that 
\[\chi(L)\ge {\rm codim}\,Z+\dim Z  +1\ge \dim A +1\>.
\] 
\end{proof}

%%%%%%%%%%%%%%

\section{GV sheaves}

\subsection{Recap on GV and M-regular sheaves} We recall from  (\ref{GV-sheaf}) that a sheaf $\F$ on abelian variety $A$  is said to be  \emph{GV} if $\Phi_{\cP^{-1}}(\F^\vee)$ is a sheaf in cohomological degree $g$, so that $\Phi_{\cP^{-1}}(\F^\vee)[g]$ coincides with the naive FMP transform $\T(F)=\Phi_{\cP^{-1}}(\F^\vee)[g]$. Therefore for a GV sheaf $\F$  the naive FMP transform  will be simply referred to as \emph{the FMP transform of $\F$}.\footnote{In the literature of GV sheaves this is usually denoted $\widehat{\F^\vee}$. In this paper we changed notation because we have been considering also non-GV sheaves, where the naive FMP transform does not coincide with $\T(F)=\Phi_{\cP^{-1}}(\F^\vee)[g]$.}  Moreover a  \emph{M-regular} sheaf is a GV sheaf such that $\T(\F)$  is torsion free.

It is also worth to recall that to be GV (resp. M-regular) is equivalent to the following condition on the cohomological support loci of $\F$ (see (\ref{Vi})): $\codim_{\Pic0 A}V^i(\F)\ge i$ (resp. $\codim_{\Pic0 A}V^i(\F)> i$) for all $i>0$ (see e.g. \cite[Theorem 2.3, Proposition 2.8]{reg3}). 

A result of Popa and the author states that a M-regular sheaf is CGG ({\bf (iii)} of Subsection \ref{intro-CGG}). This is a particular case of Corollary \ref{generated}. Indeed if $\F$ is GV then the source and target of the map ${\rm ed}_x$ of (\ref{edge}) simply coincide, and the map is the identity. Moreover, if $\F$ is, in addition, M-regular, the  FMP transform $\T(F)$ is torsion free and therefore all maps ${\rm ev}_x(U)$ of  (\ref{simultaneous}) are injective for all (non-empty) open subsets $U$ of $\Pic0 A$ and for all $x\in A$.
  It also follows that a M-regular sheaf is ample by the mentioned result of Debarre ({\bf (i)} of the Introduction). 
  
  The purpose of this section is to extend this analysis from M-regular to GV sheaves. 

\subsection{Generation of GV sheaves }  The generation of GV sheaves follows in the same way (under the usual reducedness assumption)  from Corollary \ref{generated}. 
Therefore  we have the following result, partly known by work of Popa and the author in \cite[Theorem 4.2]{reg2}. 
%Before stating it, notice again that, going back to Corollary \ref{generated}, if $\F$ is a GV sheaf all maps ${\rm ed}_x$ are just the identity. Therefore for all subsheaves $\tau$ of $\T(\F)$, the loci $\widetilde{V}^0(\tau)$ and $V^0(\tau)=\{a\in A\>|\>h^0(\tau\otimes P_x)>0\}$ coincide.

\begin{corollary}\label{GV-gen}  Let $\F$ be a GV sheaf on an abelian variety $A$. Assume that  all subsheaves   appearing in the torsion filtration of the FMP transform $\T(\F)$ have reduced scheme-theoretic support. 
Then  $\F$ is generated by the set of minimal irreducible components of the supports of such sheaves.
\end{corollary}

Alternatively,  Corollary \ref{GV-gen}  follows from Theorem \ref{D} (or Proposition \ref{naiveFMP}) combined with the following inversion result
\begin{proposition}\label{T(T)} If $\F$ is a GV sheaf then also $\T_A(\F)$ is a GV sheaf (on $\Pic0 A$) and
\[
\T_{\Pic0 A}(\T_A(\F))=\F.
\]
\end{proposition}
 Note that, in the the statement above, for sake of clarity we have denoted  $\T_A(\F)=R^g\Phi_{\cP^{-1}}^A(\F^\vee)$ and, similarly, for a coherent sheaf $\G$ on $\Pic0 A$, $\T_{\Pic0 A}(\G)=R^g\Phi^{\Pic0 A}_{\cP^{-1}}(\G^\vee)$.
  \begin{proof} By definition a sheaf $\F$ is a GV sheaf if  $\T_A(\F)=\Phi_{\cP^{-1}}^A(\F^\vee)[g]$.  Therefore, since the inverse of the FMP equivalence $ \Phi_{\cP^{-1}}^A:D(A)\rightarrow D(\Pic0 A)$ is the equivalence $\Phi_\cP^{\Pic0 A}[g]:D(\Pic0 A)\rightarrow D(A)$, it follows that
 \[\Phi_\cP^{\Pic0 A}(\T_A(\F))=\F^\vee={\bf R}\mathcal Hom(\F,\OO_A)\]
 Therefore, by Grothendieck duality (\cite[(3.8)]{mukai}), $\Phi^{\Pic0 A}_{\cP^{-1}}(\T_A(\F)^\vee)[g]=\F$.
  \end{proof}

 \begin{example}[Unipotent vector bundles]\label{unipotent} Concerning the assumption of  Corollary \ref{GV-gen}, its necessity is shown by the well known example of \emph{unipotent} vector bundles on abelian varieties. These are, by definition, vector bundles $\F$ admitting  a filtration whose successive quotients are trivial line bundles. Since a trivial bundle on an abelian variety is evidently a GV sheaf, any unipotent bundle is a GV sheaf. For such vector bundles $H^0(\F\otimes P_\alpha)\ne 0$ if and only if $\alpha=\hat 0$ (the identity point of $\Pic0 A$). Therefore the only possible set of subvarieties generating $\F$ is $\{\hat 0\}$, and this happens if and only if $\F$ is trivial, because otherwise $H^0(\F)<{\rm rk}\,\F$. In conclusion, any non-trivial unipotent vector bundle is GV but not generated. This is explained by the fact that the FMP transform $\T(\F)$ is supported at zero, but the scheme theoretic support is non-reduced, unless the sheaf $\F$ is trivial (\cite[Lemma 4.8]{semihom}). 
 Other examples obtained from this one are e.g. \emph{homogeneous} vector bundles (direct sums of unipotent vector bundles twisted by line bundles parametrized by $\Pic0 A$), and, on a product of abelian varieties $A=B\times C$, coherent sheaves of the form $p_B^*\G\boxtimes p_C^*\mathcal U$, where $\G$ is  GV on $B$ and $\mathcal U$ is homogeneous on $C$.  
 \end{example}
 
 \begin{example}[Any subset of subvarieties can be realized as a irredundant generating subset]\label{any} Corollary \ref{GV-gen}, combined with Proposition \ref{T(T)}, can be used to construct examples showing that any subset of  subvarieties is the irredundant generating set of some (locally free) sheaf. Indeed let $A$ be an abelian variety and let $\G$ be any coherent sheaf on $\Pic0 A$, such that all of its subsheaves have reduced scheme-theoretic support. Twisting with a sufficiently high power of an ample line bundle $L$ on $\Pic0 A$, by Serre vanishing we have that $V^{>0}(\G\otimes L)=\emptyset$. Therefore $\G\otimes L$ is a GV sheaf and  $\Phi_{\cP^{-1}}^{\Pic0 A}((\G\otimes L)^\vee)=\T_A(\G\otimes L)[-g]$ is locally free.  Let
 \[\F:=\T_A(\G\otimes L)
 \]
 From Proposition \ref{T(T)} it follows that  also $\F$ is a (locally free) GV sheaf and $\T_{\Pic0 A}(\F)=\G\otimes L$. Since $\G$ is arbitrary, the assertion follows from Corollary \ref{GV-gen}, because it can be easily shown, with the help of Corollary \ref{generated},  that   the generating set of minimal irreducible components of support of the sheaves appearing in the torsion filtration of $\G\otimes L$ (which coincides, after tensorization with $L^{-1}$, with the torsion filtration of $\G$) is  irredundant. 
\end{example}
 
 For the next example illustrating Corollary \ref{GV-gen}, it is useful to recall the  compatibility property of the FMP functor  with respect to homomorphisms of abelian varieties. Let  $p_B:A\rightarrow B$ be a surjective homomorphism  and $i_B : \Pic0 B\rightarrow \Pic0 A$  the dual inclusion. Then
\begin{equation}\label{compatibility}    \Phi_{\cP_A^{-1}}^A\circ p_B^*=i_{B*}\circ \Phi_{\cP_B^{-1}}^B \qquad \Phi_{\cP_A^{-1}}^{\Pic0 A}\circ i_{B*}= p_B^* \circ \Phi_{\cP^{-1}_B}^{\Pic0 B} \>.    
\end{equation}
 (This can be deduced e.g. from \cite[Proposition 1.1]{schnell}). 
 
 \begin{example}[Sheaves admitting a Chen-Jiang decomposition]\label{CJ} A significant class of GV sheaves where the assumption of Corollary \ref{GV-gen}(a) is always verified is given by \emph{coherent sheaves having the Chen-Jiang decomposition property} (\cite[\S B.3]{loposc} and references therein). They were already mentioned in Remark \ref{GV-torsion}. By definition, these sheaves admit a (essentially canonical) decomposition
\begin{equation}\label{chen-jiang}
\F=\bigoplus_i (p_{B_i}^*\G_i)\otimes P_{\alpha_i}
\end{equation}
where $p_{B_i}:A\rightarrow B_i$ are surjective homomorphisms, with connected kernel,  of abelian varieties, $\G_i$ are M-regular sheaves (hence ample) on $B_i$, and $\alpha_i\in \Pic0 A$ are torsion points. It is well known that that such sheaves are generated and semiample (this follows e.g. from {\bf (ii)} and {\bf (iii)} of Subsection \ref{intro-CGG}). This class is important because it contains higher direct images of canonical sheaves of smooth complex projective varieties mapping to abelian varieties (\cite[Theorem A]{paposc}), as well as direct images of pluricanonical sheaves (\cite[Theorem C]{loposc}), and  direct images of log-pluricanonical sheaves for klt pairs (\cite[Theorem 1.3]{M1}).

  It follows from (\ref{compatibility}) that each summand  in (\ref{chen-jiang}) is a GV sheaf, and its FMP transform is 
\[\T((p_{B_i}^*\G_i)\otimes P_{\alpha_i})=t^*_{-\alpha_i}(i_{B*}\T_B(\G_i)).
\]
In particular the various FMP transforms of the summands are \emph{scheme-theoretically} supported on the translated abelian subvarieties $\Pic0 B^i-\alpha_i$. Therefore such sheaves, as well as $\F$,  satisfy the assumption of Corollary \ref{GV-gen}. (In fact, in the proof of the above quoted results about direct images of pluricanonical sheaves, this is an essential point. Its proof revolves around the \emph{minimal extension property} of a positively curves metrics  on such direct images, whose existence was shown by Cao-P\u{a}un \cite{cao-paun}. See also the survey  \cite{haposc} on this matters.)

Therefore,  as noted in \cite[\S 9]{desperados}, the torsion filtration of the FMP transform of $\F$:  
\[T_0\subset\cdots\subset T_d\subset \T(\F),
\] where $d=\dim \T(\F)$, is given by 
\[T_k=\bigoplus_{\dim B_i\le k}t^*_{-\alpha_i}(i_{B*}\T_B(\G_i)).
\]
(Note that if $d=g(=\dim A)$, among the summands of (\ref{chen-jiang}) there is one with $B_i=A$, hence $\G_i$ is already M-regular in $A$. Its transform is the torsion free sheaf $\T(\F)/T_{g-1}$.)
\end{example}

\subsection{Structure of GV sheaves }\label{structure}  In the previous example, applying the inverse FMP transform $\Phi_{\cP}^{\Pic0 A}:D(\Pic0 A)\rightarrow D(A)$ to the torsion filtration of $\T(\F)$ one gets back the cofiltration 
\begin{equation}\label{cofilt1}\F= \F_d\twoheadrightarrow\dots \twoheadrightarrow\\F_0\rightarrow 0\>,
\end{equation}
where $F_k=\bigoplus_{\dim B_i\le k} (p_{B_i}^*\G_i)\otimes P_{\alpha_i}$. (In particular, if $d=g$ the kernel of $\F\twoheadrightarrow \F_{g-1}$ is M-regular hence ample.) This is a sort of weak version of the decomposition (\ref{chen-jiang})

The following results says that, under the reducedness assumption of Corollary \ref{GV-gen}, the structure of GV sheaves can be described by a weak analog of (\ref{cofilt1}).

\begin{theorem}\label{GV-ample} Let $\F$ be a GV sheaf on an abelian variety $A$, satisfying the assumption of Corollary \ref{GV-gen}.  Then $\F$ has a  cofiltration
\[
 \F\buildrel{\varphi_0}\over \twoheadrightarrow F_1\buildrel{\varphi_1}\over \twoheadrightarrow \cdots \buildrel{\varphi_{s-1}}\over\twoheadrightarrow F_s\buildrel{\varphi_s}\over\rightarrow 0
\]
such that: 
(a) The kernel of the surjection  $\varphi_0: \F\buildrel{\varphi_0}\over \twoheadrightarrow \F_1$ is either zero or ample, the latter case holding if and only if $\dim \T(\F)=\dim A$.\\
%(b) The map $f_s: p_{B_s}^*\F_s\otimes P_{\alpha_s}\twoheadrightarrow F_s$ is an isomorphism.  \emph{(where, if $\dim B_i=0$, all sheaves on the point $B_i$ are declared to be ample)}. \emph{(In particular $\F$ has always a quotient of the form $p_{B_i}^*\F_i\otimes P_{\alpha_i}$, with $\F_i$ ample on $B^i$.)}\\
(b) For each $i\ge 1$, the sheaf $\ker \varphi_i$ has, in turn, a cofiltration 
\[
 \ker\varphi_i\buildrel{\varphi_{0i}}\over \twoheadrightarrow F_{1i}\buildrel{\varphi_{1i}}\over \twoheadrightarrow \cdots \buildrel{\varphi_{s-1\,i}}\over\twoheadrightarrow F_{s_i\,i}\buildrel{\varphi_{s_i\,i}}\over\rightarrow 0
\]
such that for all $(j,i)$, there is a surjection $f_{j\,i}: p_{B_i}^*\F_{j\,i}\otimes P_{\alpha_{j\,i}}\twoheadrightarrow \ker \varphi_{j\,i}$, where: $p_{B_{j\,i}}:A\rightarrow B_{j\,i}$ is a surjective homomorphism of abelian varieties with connected kernel, $\F_{j\,i}$ is an ample coherent sheaf on $B_{j\,i}$ and $\alpha_{j\,i}\in\Pic0 A$
\end{theorem}

\begin{proof} We begin with a general fact, well known to the experts (see e.g. \cite[Propositions 1.1, 1.2]{schnell}).

\begin{lemma}\label{T-exchange} Let  $p:A\rightarrow B$ be a surjective homomorphism of abelian varieties, with connected kernel. Moreover let $i:\Pic0 B\hookrightarrow \Pic0 A$ be the dual inclusion,  $\tau$ a coherent sheaf on $\Pic0 B$ and $\alpha\in\Pic0 A$.  Then $\T_{\Pic0 A}(t_{-\alpha}^*i_*\tau)=p^*\T_{\Pic0 B}(\tau)\otimes P_\alpha$ \emph{here $t_\alpha$ denotes the translation by $\alpha$ on $\Pic0 A$).} 
\end{lemma}
\proof Recalling that, on an abelian variety $C$, the functor $\T_C((\cdot))$ is defined as $R^{\dim C}\Phi_{\cP^{-1}}((\cdot)^\vee)$ (see  (\ref{dual}) for the dualizing functor), the assertion of the Lemma follows from: (i) the second identity in (\ref{compatibility}), (ii) the fact that ${\bf R}{\mathcal H}om_{\Pic0 A}(i_*\tau,\OO_{\Pic0 A})=i_*{\bf R}{\mathcal H}om_{\Pic0 B}(\tau,\OO_B)[\dim A-\dim B]$, and, (iii) the well known identity $\Phi_{\cP^{-1}}^{\Pic0 A}(t_{-\alpha}^*(\cdot))=\Phi_{\cP^{-1}}^{\Pic0 A}(\cdot)\otimes P_\alpha$.  This proves Lemma \ref{T-exchange}.\endproof

Going back to the Theorem, the initial step of the cofiltration is defined as follows. Let $\tau$ be the torsion part of $\T_A(\F)$. We apply the right exact contravariant functor $\T_{\Pic0 A}(\cdot)$ to the exact sequence $0\rightarrow \tau\rightarrow \T_A(\F)\rightarrow \T_A(\F)/\tau\rightarrow 0$. By Proposition \ref{T(T)} we get the exact sequence
\[\T_{\Pic0 A}(\T_A(\F)/\tau)\longrightarrow\F\buildrel{\varphi_0}\over\longrightarrow \T_{\Pic0 A}(\tau)\rightarrow 0
\]
We define $F_1:=\T_{\Pic0 A}(\tau)$. Since the sheaf $ \T_A(\F)/\tau$ is either zero or torsion free, the sheaf $\T_{\Pic0 A}(\T_A(\F)/\tau)$ is either zero or ample (Proposition \ref{naiveFMP}). Therefore the same is true  for $\ker \varphi_0$. Hence we have accomplished the first step of the cofiltration. Next, let $d=\dim \tau$ and $0\rightarrow \tau_{d-1}\rightarrow \tau\rightarrow \tau/\tau_{d-1}\rightarrow 0$ the first step of the torsion filtration of $\tau$. Applying the functor $\T_{\Pic0 A}(\cdot)$ we get the exact sequence
\[\T_{\Pic0 A}(\tau/\tau_{d-1})\rightarrow \F_1\buildrel{\varphi_1}\over\rightarrow F_2:=\T_{\Pic0 A}(\tau_{d-1})\rightarrow 0
\]
We need to prove condition (b) on $\ker \varphi_1$. The sheaf $\sigma:=\tau/\tau_{d-1}$ has pure dimension $d$. Let us pick an irreducible component $V$ of the support of $\sigma$. We apply the functor $\T_{\Pic0 A}(\cdot )$ to the exact sequence $0\rightarrow K \rightarrow \sigma\rightarrow \sigma_{|V}\rightarrow 0$. We get 
\[\T_{\Pic0 A}(\sigma{|V})\rightarrow \T_{\Pic0 A}(\sigma)\rightarrow \T_{\Pic0 A}(K)\rightarrow 0
\]
By Proposition \ref{naiveFMP} if $V$ spans $\Pic0 A$ then the sheaf  $\T_{\Pic0 A}(\sigma_{|V})$ is ample. Otherwise $V$ is a subvariety of a translate of an abelian subvariety $C$ of $\Pic0 A$. In this case, by Lemma \ref{T-exchange}, the sheaf   $\T_{\Pic0 A}(\sigma_{|V})$ is of the form $p^*\T_{\Pic0 B}(G)\otimes P_\alpha$, where $B$ is the dual of $C$, and the sheaf $\T_{\Pic0 B}(G)$ is ample on $B$ again by Proposition \ref{naiveFMP}. This settles the first step of the cofiltration on $\ker\varphi_1$.

 Next, the sheaf $K$ is either zero or of pure dimension $d$. Therefore we can apply the same procedure to the sheaf $K$, and so on. In this way, after a finite number of steps the cofiltration on $\ker\varphi_1$ is settled. We now proceed inductively applying the same procedure to the sheaves $\tau_k$ for $k\le d-1$.
\end{proof}

%%%%%%%%%

\providecommand{\bysame}{\leavevmode\hbox
to3em{\hrulefill}\thinspace}

\end{document}